\documentclass[11pt,twoside]{article}
\usepackage[utf8]{inputenc}
\usepackage{geometry}
\usepackage{amsmath,amsthm,amssymb,chngcntr,setspace,titlesec}
\usepackage[all]{xy}

\usepackage{fancyhdr}
\newtheorem*{Main}{Main Theorem}

\newtheorem{thm}{Theorem}[section]
\newtheorem{lemma}[thm]{Lemma}

\newtheorem{prop}[thm]{Proposition}
\newtheorem{cor}[thm]{Corollary}
\newtheorem{defn}[thm]{Definition}
\newtheorem{examp}[thm]{Example}

\newtheorem{rmk}[thm]{Remark}
\numberwithin{equation}{section}
\title{A Formula for the S-Class Number of an Algebraic Torus}
\author{Minh-Hoang Tran}
\begin{document}
\date{}
\maketitle

\begin{abstract}
We obtain a formula for the S-class number of an algebraic torus defined over a number field in terms of the \'etale and Galois cohomology groups of its character module. As applications, we give different proofs of some classical class number formulas of  Shyr, Ono, Katayama and Morishita. 
\end{abstract}

\section{Introduction}
Let $T$ be an algebraic torus defined over a number field $K$ with character group $\hat{T}$. Let $S_{\infty}$ be the set of all archimedean places of $K$ and $S$ be a finite set of places of $K$ containing $S_{\infty}$. Let $O_{K,S}$ be the ring of $S$-unit integers and $U=Spec(O_{K,S})$. We write $j:Spec(K)\to U$ for the inclusion of the generic point.

For each finite prime $v$ of $K$, let $I_v$ be the inertia group of $v$. For $n\geq 1$, the Tate-Shafarevich group $\mathbb{III}^n(T)$ is defined as the kernel of the restriction map
\[ \Delta^n: H^n(K,T)\to \prod_{\mbox{all $v$}}H^n(K_v,T).\]
We also write $[M]$ for the order of a finite group $M$. Our main result is the following formula for the $S$-class number $h_{T,S}$ of $T$.
\begin{Main}
Let $T$ be an algebraic torus over a number field $K$ with character group $\hat{T}$. Let $\Psi^n(j_{*}\hat{T})$ be the map $ H^n_{et}(U,j_{*}\hat{T})\to \prod_{v\in S}H^n(K_v,\hat{T})$. Then 
\begin{eqnarray}\label{eqn_main}
 h_{T,S} &=& 
\frac{[Ext^1_U(j_{*}\hat{T},\mathbb{G}_m)][H^1(K,\hat{T})]}
{[\mathrm{ker}\Psi^1(j_{*}\hat{T})][\mathbb{III}^1(T)]\prod_{v\in S}[H^1(K_v,{T})]\prod_{v\notin S}[H^0(\hat{\mathbb{Z}},H^1(I_v,\hat{T}))]
}.
\end{eqnarray}
\end{Main}
 The class number is an important invariant of an algebraic torus. For example, it appears in the formulas for the special values of the L-function of the torus \cite{Ono61}. 
Shyr proved a formula relating two class numbers of two isogenous tori in \cite{Shy77}. Using Shyr's result, Ono \cite{Ono87} obtained a formula for the class number of a norm torus of a Galois extension. Using a different method, Katayama \cite{Kat91} found formulas for the class numbers of a norm torus and its dual for an arbitrary extension. Morishita used the techniques from Nisnevich cohomology to generalize Ono's formula to $S$-class number \cite{Mor91}. More recently, Gonzalez-Aviles  used the Nisnevich cohomology  and the result of Xarles \cite{Xar93} on the group of components of N\'eron-Raynauld models of tori to give a generalization of Chevalley's ambiguous class number formula and the Capitulation problem \cite{GA08},\cite{GA10}. Even though the Nisnevich cohomology seems more natural for the questions concerning the class numbers of tori, we work with the \'etale cohomolgy as it is more elementary and has more machinery. The proof of our Main Theorem uses homological algebra techniques together with results in Galois and \'etale cohomology, for example the Poitou-Tate exact sequence and the Artin-Verdier Duality. As applications, we will give different proofs of the formulas of  Shyr, Ono, Katayama and Morishita mentioned above. Although the following is not discussed in this paper, we cannot resist mentioning that $\eqref{eqn_main}$ will be used in a future paper to prove the following theorem.
\begin{thm}
Let $T$ be an algebraic torus defined over a number field $K$ with character group $\hat{T}$.
Suppose $S$ is a finite set of places of $K$ containing $S_{\infty}$.
We denote $h_{T,S}$, $R_{T,S}$ and $w_T$ for the $S$-class number, the $S$-regulator and the number of roots of unity of $T$ respectively.
Let $L_S(\hat{T},s)$ be the  partial Artin L-function associated with the $G_K$-representation $\hat{T}\otimes_{\mathbb{Z}}\mathbb{C}$ modulo the local factors at $S$. 
Then $\mathrm{ord}_{s=0}L_S(\hat{T},s)=\mathrm{rank}_{\mathbb{Z}}T(O_{K,S})$ and
\begin{eqnarray}
L^{*}_S(\hat{T},0)&=&\pm \frac{h_{T,S}R_{T,S}}{w_T}\frac{[\mathbb{III}^1(T)]}{[H^{1}(K,\hat{T})]} \prod_{v\in S}{[H^1(K_v,\hat{T})]}
\prod_{v \notin S}[H^0(\hat{\mathbb{Z}},H^1(I_v,\hat{T}))].
\end{eqnarray}
\end{thm}

This paper is organized as follows. In section 2, we recall some basic facts about algebraic tori and the Artin-Verdier Duality. In section 3, we obtain an exact sequence for Galois modules over local fields which is necessary for the proof of \eqref{eqn_main} given in section 4. The rest of the paper is for applications of  \eqref{eqn_main}. More precisely,  we give different proofs of the formulas of Ono, Katayama and Morishita in section 5 and generalize the formula of Shyr in section 6. 

This paper is part of my PhD thesis written at Brown University. I would like to express my gratitude to my advisor Professor Stephen Lichtenbaum for his guidance and encouragement. It is my pleasure to acknowledge the
support from the Deutsche Forschungsgesellschaft through the SFB 1085 Higher Invariant Research Group at University of Regensburg. Finally, I would like to thank Professor Guido Kings for his support and my friend Yigeng Zhao for many helpful conversations.    
\section{Preliminaries}
\subsection{Algebraic Tori}
Let $T$ be an algebraic torus defined over a number field $K$ with character group $\hat{T}$.  For each place $v$ of $K$, let $T_v$ be the base extension of $T$ to $K_v$ and let $\hat{T}_v$ be the character group of $T_v$. We define $T_v^c$ to be 
 \[ T_v^c:=\{x \in T_v(K_v) \text{ such that for all } \chi \in H^0(K_v,\hat{T_v}) , |\chi(x)|_{v}=1 \}. \]
Then $T_v^c$ is the unique maximal compact subgroup of $T_v(K_v)$.
If $v$ is a finite place then $T_v^c\simeq Hom_{K_v}(\hat{T_v},O_v^{*})$ $\cite[\mbox{page 115-116}]{Ono61}$. The following definitions are taken from \cite{Ono61}.
\begin{defn}
	Let $S$ be a finite set of places of $K$ which contains $S_{\infty}$. Define
	\[ T_{\mathbb{A},S}:=\prod_{v \in S} T_v(K_v) \times \prod_{v \notin S} T_v^c \qquad \& \qquad
T_{\mathbb{A}}:=\varinjlim_{\text{S finite}}T_{\mathbb{A},S} \]
\end{defn}
\begin{defn}
Let $S$ be a finite set of places of $K$ containing $S_{\infty}$. Define
\begin{enumerate}
 \item $ T(O_{K,S}):= T_{\mathbb{A},S} \cap T(K) $ - the S-units of $T$. 
 \item $ Cl_{S}(T):=T_{\mathbb{A}}/T(K)T_{\mathbb{A},S} $ - the S-class group of T. 
\end{enumerate}
\end{defn}
Then $T(O_{K,S})$ is finitely generated and $Cl_S(T)$ is a finite abelian group. We denote $h_{T,S}$ for the order of $Cl_S(T)$, it is also called the $S$-class number of $T$. 
\begin{prop}\label{prop_unit_class_seq}
We have the exact sequences 
\begin{equation}\label{seq_unit_class_gp}
0\to T(O_{K,S}) \to T(K) \to \coprod_{v\notin S}{T(K_v)/T_v^c} \to Cl_S(T) \to 0,
\end{equation} 
\begin{equation}\label{seq_unit_class_gp_2}
0\to T(O_K) \to T(O_{K,S}) \to \prod_{v\in S-S_{\infty}}{T(K_v)/T_v^c} \to Cl(T) \to Cl_S(T) \to 0.
\end{equation} 
In particular, $T(O_{K,S})_{tor}\simeq T(O_K)_{tor}$. 
\end{prop}
\begin{proof}
Note that $T_{\mathbb{A}}/T_{\mathbb{A},S} \simeq \coprod_{v\notin S}{T(K_v)/T_v^c}$. Then ($\ref{seq_unit_class_gp}$) follows immediately. To prove \eqref{seq_unit_class_gp_2}, we only need to apply the Snake Lemma to the following commutative diagram 
\[ \xymatrix{  0 \ar[r] & 0 \ar[r] \ar[d] & T(K) \ar[r] \ar[d] & T(K) \ar[r] \ar[d] & 0\\
            0 \ar[r] & \prod_{v\in S-S_{\infty}} \frac{T(K_v)}{T_v^{c}} \ar[r] & 
  \prod_{v\notin S_{\infty}} \frac{T(K_v)}{T_v^{c}} \ar[r] & 
 \prod_{v\notin S} \frac{T(K_v)}{T_v^{c}} \ar[r] & 0 \\}
\]
and invoke \eqref{seq_unit_class_gp} for $S$ and $S_{\infty}$. The last statement follows from \eqref{seq_unit_class_gp_2} and the fact that ${T(K_v)/T_v^c}$ is torsion free. 
\end{proof}
\subsection{Artin-Verdier Duality}
Let $K$ be a number field with Galois group $G_K$ and $X=Spec(O_K)$. Let $U$ be an open subscheme of $X$.  For any sheaf $\mathcal{F}$ on $U$, let $\mathcal{F}_v$ be the unique discrete $G_{K_v}$-module corresponding to the pull-back of $\mathcal{F}$ to $Spec(K_v)$. In \cite[\mbox{page 165}]{Mil06}, Milne  constructed the cohomology groups with compact support $H^r_c(U,\mathcal{F})$ which satisfies the following long exact sequence
\begin{equation}\label{seq_compact_coh}
... \to H^r_c(U,\mathcal{F})\to H^r_{et}(U,\mathcal{F}) \to \bigoplus_{v \notin U} H^r(K_v,\mathcal{F}_v) \to ...
\end{equation}
where $H^r(K_v,\mathcal{F}_v)$ is the usual Galois cohomology if $v$ is a finite prime and is the Tate cohomology if $v$ is an archimedean prime. 
For an abelian group $A$, let $\hat{A}$ be the profinite completion of $A$ and $A^D:=Hom_{\mathbb{Z}}(A,\mathbb{Q}/\mathbb{Z})$. The following theorem will be used many times in this paper.
\begin{thm}[Artin-Verdier Duality]\label{Artin-Verdier}
There is a canonical isomorphism $H^3_{c}(U,\mathbb{G}_m) \simeq \mathbb{Q}/\mathbb{Z}$. 
Furthermore, for any $\mathbb{Z}$-constructible sheaf $\mathcal{F}$ on $U$, the pairing 
	\[ H^{3-r}_{c}(U,\mathcal{F}) \times Ext^{r}_{U}(\mathcal{F},\mathbb{G}_m) \to H^3_{c}(U,\mathbb{G}_m) \]
	induces a map 
$\alpha^r(\mathcal{F}) : Ext^{r}_{U}(\mathcal{F},\mathbb{G}_m) \to H^{3-r}_{c}(U,\mathcal{F})^D $
which satisfies
		\begin{enumerate}
			\item For $r=0,1$, $Ext^{r}_{U}(\mathcal{F},\mathbb{G}_m)$ is finitely generated and $\alpha^r(\mathcal{F})$
			induces an isomorphism of profinite groups
\[ \hat{\alpha}^r(\mathcal{F}): \widehat{Ext^{r}_{U}(\mathcal{F},\mathbb{G}_m)} \to H^{3-r}_{c}(U,\mathcal{F})^D. \]
			\item For $r=2,3$, 
			$\alpha^r(\mathcal{F})$ is an isomorphism between groups of cofinite type.
		\end{enumerate}
Furthermore, if $\mathcal{F}$ is constructible then $\alpha^r(\mathcal{F})$ is an isomorphism of finite groups for all $r$.
	\end{thm}
	\begin{proof}
		See $\cite[\mbox{II.3.1}]{Mil06}$.
	\end{proof}
\begin{defn}\label{defn_negligible}
We say $\mathcal{F}$ is a negligible sheaf on $U$ if $\mathcal{F}$ has finite support and its stalks are finite everywhere. Note that negligible sheaves are constructible.
\end{defn}
\begin{lemma}
Let $\mathcal{F}$ be a negligible sheaf on $U$. Then 
\begin{enumerate}
\item $H^n_c(U,\mathcal{F})\simeq H^n_{et}(U,\mathcal{F})$ for all $n$.
\item $H^n_{et}(U,\mathcal{F})=0$ for $n>1$ and $Ext^n_{U}(\mathcal{F},\mathbb{G}_m)=0$ for $n=0,1$. 
\item $[H^0_{et}(U,\mathcal{F})]=[H^1_{et}(U,\mathcal{F})]$ and $[Ext^2_{U}(\mathcal{F},\mathbb{G}_m)]=[Ext^3_{U}(\mathcal{F},\mathbb{G}_m)]$.
\end{enumerate}
\end{lemma}
\begin{proof}
It suffices to assume $\mathcal{F}=i_{*}M$ where $i:p \to U$ is the immersion of a closed point $p$ of $U$ and $M$ is a finite discrete $\hat{\mathbb{Z}}$-module. As the generic stalk of $i_{*}M$ vanishes, $H^n_c(U,i_{*}M)\simeq H^n_{et}(U,i_{*}M)$ by ($\ref{seq_compact_coh}$).  Since $H^n_{et}(U,i_{*}M)\simeq H^n(\hat{\mathbb{Z}},M)$ and $\hat{\mathbb{Z}}$ has cohomological dimension 2, $H^n_{et}(U,i_{*}M)=0$ for $n>1$. Hence, by Theorem $\ref{Artin-Verdier}$, $Ext^n_{U}(i_{*}M,\mathbb{G}_m)=0$ for $n=0,1$. Finally, $[H^0(\hat{\mathbb{Z}},M)]=[H^1(\hat{\mathbb{Z}},M)]$ by $\cite[\mbox{page 32}]{Mil06}$. 
\end{proof}
\begin{examp}
Let $N$ be a discrete $G_K$-module, torsion free as an abelian group. Then any subsheaf of $R^1j_{*}N$ is negligible. Indeed, the generic stalk of $R^1j_{*}N$ vanishes so it only has finite support. Moreover, if $\bar{p}$ is a geometric point over a closed point $p$ of $U$ then $(R^1j_{*}N)_{\bar{p}}\simeq H^1(I_p,N)$ which is finite.
\end{examp}
\begin{prop}\label{prop_finiteness_jN}
Let $M$ be a discrete $G_K$-module, torsion free and finitely generated as an abelian group. Then 
\begin{enumerate}
\item $H^0_{et}(U,j_{*}M)$, $H^0_{c}(U,j_{*}M)$, $H^1_{c}(U,j_{*}M)$ and $Hom_U(j_{*}M,\mathbb{G}_m)$ are finitely generated.
\item $H^1_{et}(U,j_{*}M)$, $H^2_{c}(X,j_{*}M)$ and $Ext^1_U(j_{*}M,\mathbb{G}_m)$ are finite abelian groups.
\item $Ext^n_U(j_{*}M,\mathbb{G}_m)$ are of cofinite type for $n=2,3$. 
\end{enumerate}
\end{prop}
\begin{proof}
By $\cite[\mbox{1.5.1}]{Ono61}$, there exist finitely many Galois extensions $\{K_{\mu}\}_{\mu}$, $\{K_{\lambda}\}_{\lambda}$ of $K$, a positive integers $n$,
and a finite $G_K$-module $N$ such that we have the exact sequence 
\begin{equation}\label{ono_etale_seq1}
0 \to M^{n}\oplus \prod_{\mu}(\pi_{\mu})_{*}\mathbb{Z} \to \prod_{\lambda}(\pi_{\lambda})_{*}\mathbb{Z} \to N \to 0,
\end{equation}
where $\pi_{\mu}$ is the map $ Spec(K_{\mu})\to 
Spec(K)$. Let $V_{\mu}$ be the normalization of $U$ in $K_{\mu}$ and $\pi_{\mu}'$ be the map $V_{\mu} \to U$. Applying $j_{*}$ to ($\ref{ono_etale_seq1}$) we obtain 
\begin{equation}\label{ono_etale_seq2}
0 \to (j_{*}M)^{n}\oplus \prod_{\mu}(\pi_{\mu}')_{*}\mathbb{Z} \to \prod_{\lambda}(\pi_{\lambda}')_{*}\mathbb{Z} \to \mathcal{Q} \to 0
\end{equation}
where $\mathcal{Q}$ is a constructible sheaf. From $\cite[\mbox{II.2.1}]{Mil06}$, 
$Ext^{n}_U((\pi_{\mu}')_{*}\mathbb{Z},\mathbb{G}_m)\simeq H^{n}_{et}(V_{\mu},\mathbb{G}_m)$ is finitely generated, finite, cofinite type for $n=0,1$ and $2,3$ respectively. Therefore, by the long exact sequence of the $Ext$-groups associated to ($\ref{ono_etale_seq2}$), we obtain the similar statement for $Ext^n_U(j_{*}M,\mathbb{G}_m)$. As $Ext^1_U(j_{*}M,\mathbb{G}_m)$ is finite, by Artin-Verdier duality, $H^2_c(U,j_{*}M)$ is also finite. 

As $H^0_{et}(U,j_{*}M)\simeq H^0(K,M)$, it is finitely generated. By the long exact sequence of $H^n_{et}(U,.)$ associated with ($\ref{ono_etale_seq2}$) and the fact that $H^1_{et}(U,\mathbb{Z})=0$, we deduce that $H^1_{et}(U,j_{*}M)$ is finite. Finally, we have the exact sequence from ($\ref{seq_compact_coh}$)
\[ 0 \to \prod_{v\in S}H^{-1}(K_v,M) \to H^0_c(U,j_{*}M) \to H^0(K,M) 
\to \prod_{v\in S}H^{0}(K_v,M) \to H^1_c(U,j_{*}M) \to H^1_{et}(U,j_{*}M) 
\]
Since $\prod_{v\in S}H^{n}(K_v,M)$ is finitely generated for all $n$ (note that for $v\in S_{\infty}, H^n(K_v,M)$ means the Tate cohomology group), $H^n_c(U,j_{*}M)$ is finitely generated for $n=0,1$.
\end{proof}

\section{Local Galois Cohomology}
Let $K$ be a number field with Galois group $G_K$. 
Let $U$ be an open subscheme of $Spec(O_K)$ and $j:Spec(K)\to U$.
For each finite place $v$ of $K$, we write $K_v^{ur}$ for the maximal unramified extension of the completion $K_v$ of $K$. We denote by $I_v=G(\bar{K}_v/K_v^{ur})$ the inertia group of $v$. Let $O_{v}$, $O_{v}^{ur}$ and $\bar{O}_{v}$ be the valuation rings of $K_v$, $K_v^{ur}$ and $\bar{K}_v$ respectively. For the rest of this paper, by a discrete $G_K$-module, we mean a finitely generated abelian group with continuous $G_K$-action.
\begin{lemma}\label{lemma_localext}
Let $\mathcal{F}$ be a sheaf on $U$ and $\mathcal{F}_K$ be the $G_K$-module corresponding to the sheaf $j^{*}\mathcal{F}$ on $Spec(K)$. Let $\hat{\mathcal{F}}_K=Hom_{\mathbb{Z}}(\mathcal{F}_K,\bar{K}^{*})$. Then 
\begin{enumerate}
\item $Ext_U^{n}(\mathcal{F},j_{*}\mathbb{G}_m) \simeq H^n(K,\hat{\mathcal{F}}_K)$. 
In particular,
$ Ext_U^{n}(j_{*}\hat{T},j_{*}\mathbb{G}_m) \simeq H^{n}(K,T)$.
\item Let $i:v\to U$ be a closed immersion. Then $Ext_U^{n}(\mathcal{F},i_{*}\mathbb{Z}) \simeq Ext^{n}_{\hat{\mathbb{Z}}}(\mathcal{F}_K^{I_v},\mathbb{Z})$. In particular, 
$Ext_U^{n}(j_{*}\hat{T},i_{*}\mathbb{Z}) \simeq Ext^{n}_{\hat{\mathbb{Z}}}(\hat{T}^{I_v},\mathbb{Z})$.
\end{enumerate}
\end{lemma}
\begin{proof}
\begin{enumerate}
\item From $\cite[\mbox{II.1.4}]{Mil06}$, $R^qj_{*}\mathbb{G}_m=0$ for $q>0$. Thus the spectral sequence 
\[ Ext^p_U(\mathcal{F},R^qj_{*}\mathbb{G}_m) \Rightarrow Ext^{p+q}_{G_K}(\mathcal{F}_K,\bar{K}^{*})\] 
collapses and yields $ Ext_U^{p}(\mathcal{F},j_{*}\mathbb{G}_m) \simeq Ext^{p}_{G_K}(\mathcal{F}_K,\bar{K}^{*})$. 
From $\cite[\mbox{I.0.8}]{Mil06}$, there is a spectral sequence 
\[ H^p(G_K,Ext^q_{\mathbb{Z}}(\mathcal{F}_K,\bar{K}^{*})) \Rightarrow Ext^{p+q}_{G_K}(\mathcal{F}_K,\bar{K}^{*}). \]
Since $\bar{K}^{*}$ is divisible, $Ext^q_{\mathbb{Z}}(N,\bar{K}^{*})=0$ for $q>0$.
Thus,
$H^p(G_K,Hom_{\mathbb{Z}}(\mathcal{F}_K,\bar{K}^{*}))\simeq Ext^{p}_{G_K}(\mathcal{F}_K,\bar{K}^{*})$. 
Hence, $Ext_U^{n}(\mathcal{F}_K,j_{*}\mathbb{G}_m)\simeq H^n(K,\hat{\mathcal{F}}_K)$.
Finally, if $\mathcal{F}=j_{*}\hat{T}$ then $\hat{\mathcal{F}}_K=Hom_{\mathbb{Z}}(\hat{T},\bar{K}^{*}) \simeq T(\bar{K})$. Therefore, $ Ext_U^{n}(j_{*}\hat{T},j_{*}\mathbb{G}_m) \simeq H^{n}(K,T)$.

\item  Since $i_{*}$ is exact, the spectral sequence 
\[ Ext^n_U(\mathcal{F},R^mi_{*}\mathbb{Z}) \Rightarrow 
Ext^{n+m}_{v}(i^{*}\mathcal{F},\mathbb{Z})\simeq Ext^{n+m}_{\hat{\mathbb{Z}}}(\mathcal{F}_K^{I_v},\mathbb{Z}) \]
implies that
$Ext_U^{n}(\mathcal{F},i_{*}\mathbb{Z}) \simeq Ext^{n}_{\hat{\mathbb{Z}}}(\mathcal{F}_K^{I_v},\mathbb{Z})$.
\end{enumerate}
\end{proof}

\begin{prop}\label{local_duality_fg}
Let $N$ be a discrete $G_{K_v}$-module. Let $\hat{N}=Hom_{\mathbb{Z}}(N,\bar{K}_{v}^{*})$ and $\hat{N}^c=Hom_{\mathbb{Z}}(N,\bar{O}_{v}^{*})$. Then
\begin{enumerate}
\item $ H^0(K_v,\hat{N}^c) = 
\{ f \in Hom_{G_{K_v}}(N,\bar{K}_v^{*}) : \text{ for all } x \in H^0(K_v,N) \text{, we have } f(x) \in O_v^{*} \} $.
\item $ Ext^{2}_{\hat{\mathbb{Z}}}(N^{I_v},\mathbb{Z}) \simeq H^2(K_v,\hat{N})$.
\item The following sequence is exact
\begin{multline}\label{local_duality_sq}
0 \to H^0(K_v,\hat{N}^c) \to H^0(K_v,\hat{N}) \to Hom_{\hat{\mathbb{Z}}}(N^{I_v},\mathbb{Z}) \to
 H^0(\hat{\mathbb{Z}},H^1(I_v,N))^{D} \to \\
 \to H^1(K_v,\hat{N}) \to Ext^{1}_{\hat{\mathbb{Z}}}(N^{I_v},\mathbb{Z}) \to 0.
\end{multline}
\end{enumerate} 
\end{prop}

\begin{proof}
\begin{enumerate}
\item Let $f \in H^0(K_v,\hat{N}^c)$. For any $x \in H^0(K_v,N)$, $f(x) \in \bar{O}_{v}^{*}$ by definition. As $f$ is $G_{K_v}$-invariant, $f(x) \in O_v^{*}$. Thus, $H^0(K_v,\hat{N}^c)$ is a subset of the right hand side.

Conversely, let $f$ be an element of the right hand side. Let $L_v$ be a finite Galois extension of $K_v$ such that the Galois group $G_{L_v}$ acts trivially on $N$. For $x\in N$, $f(x) \in L_v^{*}$ as $N=H^0(L_v,N)$. We have $ N_{L_v/K_v}(f(x))=f(Tr_{L_v/K_v}(x))$.
As $Tr_{L_v/K_v}(x) \in H^0(K_v,N)$, $f(Tr_{L_v/K_v}(x)) \in O_v^{*}$. 
Hence, $N_{L_v/K_v}(f(x)) \in O_v^{*}$. We deduce that $f(x)\in O_{L_v}^{*} \subset \bar{O}_v^{*}$. As a result, the right hand side is a subset of $ H^0(K_v,\hat{N}^c)$.

\item By $\cite[\mbox{I.1.10}]{Mil06}$,
$Ext^{2}_{\hat{\mathbb{Z}}}(N^{I_v},\mathbb{Z}) \simeq H^0(\hat{\mathbb{Z}},N^{I_v})^D\simeq H^0(K_v,N)^D$. From $\cite[\mbox{I.2.1}]{Mil06}$, we have $H^2(K_v,\hat{N}) \simeq H^0(K_v,N)^D $. 
Thus, $H^2(K_v,\hat{N}) \simeq Ext^{2}_{\hat{\mathbb{Z}}}(N^{I_v},\mathbb{Z}) $.

\item From the spectral sequence $ H^r(\hat{\mathbb{Z}},H^s(I_v,N))\Rightarrow H^{r+s}(K_v,N)$, 
we obtain
\[0 \to H^1(\hat{\mathbb{Z}},N^{I_v}) \to H^{1}(K_v,N) \to H^0(\hat{\mathbb{Z}},H^1(I_v,N)) \to H^2(\hat{\mathbb{Z}},N^{I_v}) 
\to H^{2}(K_v,N)\]
Taking Pontryagin dual and use Tate's local duality Theorem, we have
\begin{equation}\label{local_duality_sequence1}
 \widehat{H^0(K_v,\hat{N})} \xrightarrow{\hat{\Psi}} \widehat{Hom_{\hat{\mathbb{Z}}}(N^{I_v},\mathbb{Z})} \to H^0(\hat{\mathbb{Z}},H^1(I_v,N))^D 
\to H^1(K_v,\hat{N}) \to Ext^{1}_{\hat{\mathbb{Z}}}(N^{I_v},\mathbb{Z}) \to 0. 
\end{equation} 

Let $W=\mathrm{cok}\hat{\Psi}$. As $H^0(\hat{\mathbb{Z}},H^1(I_v,N))$ is finite, so is $W$. To complete the proof, we shall show the following sequence is exact.
\[ 0 \to H^0(K_v,\hat{N}^c)  \to H^0(K_v,\hat{N}) \xrightarrow{\Psi} Hom_{\hat{\mathbb{Z}}}(N^{I_v},\mathbb{Z}) 
\to W \to 0.\]
The map $\Psi$ is defined as follows : for $f \in H^0(K_v,\hat{N})$ and $x \in N^{I_v}$, $\Psi(f)(x):=v(f(x))$ where $v$ is the normalized valuation of $K_v$.
Then $\Psi$ is a continuous map and 
\[ \ker\Psi=\{f \in Hom_{G_{K_v}}(N,\bar{K}_v^{*}) : \text{ for all } x \in N^{I_v} \text{, we have } f(x) \in (O_v^{ur})^{*}\}. \]

\textbf{Claim :}  $\ker \Psi = H^0(K_v,\hat{N}^c)$.

\textbf{Proof of claim :}
\begin{itemize}
\item Let $f\in H^0(K_v,\hat{N}^c)$ and $x\in N^{I_v}$. Then 
$f(x) \in H^0(I_v,\bar{K}^{*})\cap \bar{O_v}^{*}=(O_v^{ur})^{*}$.
Therefore, $H^0(K_v,\hat{N}^c) \subset \ker \Psi$.
\item To prove the other inclusion, we use the description of $H^0(K_v,\hat{N}^c)$ from part 1. 
Let $f \in \ker \Psi$ and $x \in H^0(K,N)$. Then $f(x)\in (O_v^{ur})^*$ by definition. Since $f(x)\in K_v^{*}$, $f(x) \in (O_v^{ur})^* \cap K_v^{*}=O_v^{*}$. Hence, $\ker \Psi \subset H^0(K_v,\hat{N}^c)$.
\end{itemize}

\end{enumerate}
Let $W'=\mathrm{cok}(\Psi)$. We have the following exact sequence where all the maps are strict morphisms 
$\cite[\mathrm{page 13}]{Mil06}$
\[ 0 \to H^0(K_v,\hat{N}^c)  \to H^0(K_v,\hat{N}) \xrightarrow{\Psi} Hom_{\hat{\mathbb{Z}}}(N^{I_v},\mathbb{Z}) \to W' \to 0. \]
As profinite completion is exact for sequences with strict morphisms $\cite[\mathrm{page 14}]{Mil06}$,
\[ 0 \to \widehat{H^0(K_v,\hat{N}^c)}  \to \widehat{H^0(K_v,\hat{N})} \xrightarrow{\hat{\Psi}} \widehat{Hom_{\hat{\mathbb{Z}}}(N^{I_v},\mathbb{Z})} \to \hat{W'} \to 0 .\]
As $H^0(K_v,\hat{N}^c)$ is compact and totally disconnected (topologically it is a product of finitely many copies of $O_v^{*}$), it is a profinite group. Therefore, $\widehat{H^0(K_v,\hat{N}^c)}={H^0(K_v,\hat{N}^c)}$. Moreover, $\hat{W'}=W$ which is a finite group. Hence $W'=W$. That completes the proof of the proposition.
\end{proof}

\begin{examp}\label{local_duality_examp}
\begin{enumerate}
\item Let $N=\mathbb{Z}$. Then $\hat{N}=\bar{K}_v^{*}$ and $\hat{N}^c=\bar{O}_v^{*}$. We have $H^1(I_v,\mathbb{Z})=0$, $H^1(K_v,\mathbb{G}_m)=0$ and $Ext^1_{\hat{\mathbb{Z}}}(\mathbb{Z},\mathbb{Z})=0$. In this case, ($\ref{local_duality_sq}$) is no other than
	\[ 0 \to O_v^{*} \to K_v^{*} \to \mathbb{Z} \to 0. \]
	\item Let $N=\mathbb{Z}/n$. We have $\widehat{\mathbb{Z}/n} \simeq \mu_n$. Therefore, ($\ref{local_duality_sq}$) becomes
	\begin{equation}\label{local_duality_finite_sq}
	0 \to H^0(\hat{\mathbb{Z}},H^1(I_v,\mathbb{Z}/n))^D \to H^1(K_v,\mu_n) \to Ext^{1}_{\hat{\mathbb{Z}}}(\mathbb{Z}/n,\mathbb{Z}) \to 0.
	\end{equation}
We have $H^1(K_v,\mu_n) \simeq K_v^{*}/(K_v^{*})^n$ and $Ext^{1}_{\hat{\mathbb{Z}}}(\mathbb{Z}/n,\mathbb{Z}) \simeq \mathbb{Z}/n$. By $\cite[\mbox{page 144}]{CF10}$, $I_v^{ab} \simeq O_v^{*}$. Thus, $H^2(I_v,\mathbb{Z})\simeq H^1(I_v,\mathbb{Q}/\mathbb{Z})\simeq I_v^D \simeq (O_v^{*})^D$. Moreover, $H^1(I_v,\mathbb{Z})=0$. Hence, $H^1(I_v,\mathbb{Z}/n) \simeq (O_v^{*}/(O_v^{*})^n)^D$. Therefore, $(\ref{local_duality_finite_sq})$ reduces to
\[ 0 \to O_v^{*}/(O_v^{*})^n \to K_v^{*}/(K_v^{*})^n \to \mathbb{Z}/n \to 0. \]
\item Let $N=\hat{T}$ for some torus $T$ over $K_v$. Then $H^0(K_v,\hat{N})=T(K_v)$, 
	$H^0(K_v,\hat{N}^c)=T^c_v$, the maximal compact subgroup of $T(K_v)$ and 
\begin{equation}\label{local_duality_seq_2}
 0 \to \frac{T(K_v)}{T^c_v} \to Hom_{\hat{\mathbb{Z}}}(\hat{T}^{I_v},\mathbb{Z}) 
\to H^0(\hat{\mathbb{Z}},H^1(I_v,\hat{T}))^D \to H^1(K_v,T) \to Ext^{1}_{\hat{\mathbb{Z}}}(\hat{T}^{I_v},\mathbb{Z}) \to 0.
\end{equation}
\end{enumerate}
\end{examp}
\section{Proof of the Main Theorem}
Let $T$ be an algebraic torus over a number field $K$ with character group $\hat{T}$. Let $S$ be a finite set of places of $K$ containing $S_{\infty}$.
For $n\geq 1$, let $\mathbb{III}^n_{S}(T)$ be the kernel of the restriction map
\[  \Delta^n_{S} : H^n(K,T)\to \prod_{v \notin S}H^n(K_v,T).\]
\begin{lemma} \label{lemma_Sha_infty}
We have the formula
\begin{equation}\label{eqn_main_lemma}
 \frac{[\mathrm{cok}\Delta^1_{S}]}{[\mathbb{III}^1_{S}(T)]}=
\frac{[H^1(K,\hat{T})]}
{[\mathbb{III}^1(T)][\mathbb{III}^2(T)]\prod_{v\in S}[H^1(K_v,T)]}.
\end{equation}
\end{lemma}
\begin{proof}
We have the following commutative diagram
\[  \xymatrix{   0 \ar[r] & 0 \ar[r] \ar[d] & H^n(K,T) \ar[d]^{\Delta^n} \ar[r] & H^n(K,T) \ar[r] \ar[d]^{\Delta^n_{S}} & 0 \\
0 \ar[r] &\prod_{v \in S}H^n(K_v,T) \ar[r]  & \prod_{v}H^1(K_v,T) \ar[r]& \prod_{v \notin S}H^n(K_v,T)  \ar[r] &0  .}
\]
Applying the Snake lemma to the above diagram yields the exact sequence
\begin{equation}\label{seq_main_lemma_1}
0 \to \mathbb{III}^n(T) \to \mathbb{III}^n_{S}(T) \to \prod_{v \in S}H^n(K_v,T)
\to \mathrm{cok}\Delta^n \to \mathrm{cok}\Delta^n_{S} \to 0.
\end{equation}
Now we recall the Poitou-Tate exact sequence  $\cite[\mbox{I.4.20}]{Mil06}$,
\begin{multline}\label{Poitou_Tate}
0 \to \mathbb{III}^1(T) \to H^1(K,T) \xrightarrow{\Delta^1} \prod_{v}H^1(K_v,T)\to \\
 \to H^1(K,\hat{T})^{D} \to H^2(K,T) \xrightarrow{\Delta^2} \prod_{v}H^2(K_v,T) \to 
 H^0(K,\hat{T})^D \to 0.
\end{multline}
From ($\ref{Poitou_Tate}$), we deduce that $\mathrm{cok}\Delta^2\simeq H^0(K,\hat{T})^D$ and $[H^1(K,\hat{T})]=[\mathrm{cok}\Delta^1][\mathbb{III}^2(T)]$. 
Note that for $n=1$, ($\ref{seq_main_lemma_1}$) is an exact sequence of finite groups. Therefore, ($\ref{eqn_main_lemma}$) follows from  ($\ref{seq_main_lemma_1}$) and the fact that $[H^1(K,\hat{T})]=[\mathrm{cok}\Delta^1][\mathbb{III}^2(T)]$.
\end{proof}

\begin{thm}\label{thm_hTS}
Let $T$ be an algebraic torus over a number field $K$ with character group $\hat{T}$. Let $S$ be a finite set of places of $K$ containing $S_{\infty}$. Let $j:Spec(K)\to U=Spec(O_{K,S})$ be the inclusion of the generic point. Let $\Psi^n(j_{*}\hat{T})$ be the map $ H^n_{et}(U,j_{*}\hat{T})\to \prod_{v\in S}H^n(K_v,\hat{T})$. Then 
\begin{eqnarray}\label{eqn_main1}
 h_{T,S} = 
\frac{[Ext^1_U(j_{*}\hat{T},\mathbb{G}_m)][H^1(K,\hat{T})]}
{[\mathrm{ker}\Psi^1(j_{*}\hat{T})][\mathbb{III}^1(T)]\prod_{v\in S}[H^1(K_v,{T})]\prod_{v\notin S}[H^0(\hat{\mathbb{Z}},H^1(I_v,\hat{T}))]
}.
\end{eqnarray}
\end{thm}
\begin{proof}
From the short exact sequence of \'etale sheaves on $U$
\[ 0\to \mathbb{G}_m \to j_{*}\mathbb{G}_m \to \coprod_{v\notin S}i_{*}\mathbb{Z} \to 0 \]
we obtain the long exact sequence
\begin{equation}\label{value_galois_seq0}
 ...\to Ext_U^n(j_{*}\hat{T},\mathbb{G}_m) \to  Ext_U^n(j_{*}\hat{T},j_{*}\mathbb{G}_m) \to 
 \prod_{v \notin S}Ext_U^n(j_{*}\hat{T},i_{*}\mathbb{Z}) \to ...
\end{equation}
By Lemma $\ref{lemma_localext}$ and Proposition $\ref{local_duality_fg}$, ($\ref{value_galois_seq0}$) can be rewritten as
\begin{equation}\label{value_galois_seq1}
...\to Ext_U^n(j_{*}\hat{T},\mathbb{G}_m) \to H^n(K,T) \xrightarrow{\Theta^n_S} 
 \prod_{v \notin S}Ext^{n}_{\hat{\mathbb{Z}}}(\hat{T}^{I_v},\mathbb{Z}) \to ... 
\end{equation}
From ($\ref{value_galois_seq1}$), we note that $\ker(\Theta_S^0)\simeq Hom_U(j_{*}\hat{T},\mathbb{G}_m)$ and
\begin{equation}\label{value_galois_seq4}
0 \to \mathrm{cok}(\Theta^0_S) \to Ext_U^1(j_{*}\hat{T},\mathbb{G}_m) 
\to \ker(\Theta^1_S) \to 0 ,
\end{equation}
\begin{equation}\label{value_galois_seq6}
 0 \to \mathrm{cok}(\Theta^1_S) \to Ext_U^2(j_{*}\hat{T},\mathbb{G}_m) \to 
 \mathbb{III}^{2}_{S}(T) \to 0 .
\end{equation}
Note that $\mathrm{cok}(\Theta^0_S)$, $\ker(\Theta^1_S)$ are finite groups as $Ext_U^1(j_{*}\hat{T},\mathbb{G}_m)$ is finite. 
For each finite prime $v$ of $K$, we split the sequence $(\ref{local_duality_seq_2})$ into
\begin{equation}\label{value_galois_seq9}
 0 \to T(K_v)/T_v^{c} \to Hom_{\hat{\mathbb{Z}}}(\hat{T}^{I_v},\mathbb{Z}) \to S_v \to 0 ,
\end{equation}
\begin{equation}\label{value_galois_seq10}
 0 \to S_v \to H^0(\hat{\mathbb{Z}},H^1(I_v,\hat{T}))^{D} \to H^1(K_v,T) 
\to Ext^{1}_{\hat{\mathbb{Z}}}(\hat{T}^{I_v},\mathbb{Z}) \to 0 .
\end{equation}
From ($\ref{value_galois_seq9}$), we obtain the following commutative diagram
\[ \xymatrix{  0 \ar[r] & T(K) \ar[r] \ar[d] & T(K) \ar[r] \ar[d]^{\Theta^0_S} & 0 \ar[r] \ar[d] & 0\\
            0 \ar[r] & \prod_{v\notin S} \frac{T(K_v)}{T_v^{c}} \ar[r] & 
  \prod_{v\notin S}Hom_{\hat{\mathbb{Z}}}(\hat{T}^{I_v},\mathbb{Z}) \ar[r] & 
  \prod_{v\notin S}{S_v} \ar[r] & 0 \\}
\]
The Snake Lemma combining with ($\ref{seq_unit_class_gp}$) yield $T(O_{K,S}) \simeq Hom_U(j_{*}\hat{T},\mathbb{G}_m) $ and 
\begin{equation}\label{value_galois_seq7}
 0 \to Cl_S(T) \to \mathrm{cok}(\Theta^0_S) \to \prod_{v \notin S}S_v \to 0 .
\end{equation}
As $\mathrm{cok}(\Theta^0_S)$ is finite, ($\ref{value_galois_seq7}$) implies 
$\prod_{v \notin S}S_v$ is finite and 
\begin{equation}\label{value_galois_seq8}
h_{T,S}=\frac{[ \mathrm{cok}(\Theta^0_S)]}{\prod_{v \notin S}[S_v]}
=\frac{[Ext_U^1(j_{*}\hat{T},\mathbb{G}_m)]}
{[\ker(\Theta^1_S)]\prod_{v \notin S}[S_v]}.
\end{equation}
From ($\ref{value_galois_seq10}$), we have
\[  \xymatrix{   0 \ar[r] & 0 \ar[r] \ar[d] & H^1(K,T) \ar[d]^{\Delta^1_{S}} \ar[r] & H^1(K,T) \ar[r] \ar[d]^{\Theta^1_S} & 0 \\
   0 \ar[r] & \prod_{v\notin S} \frac{H^0(\hat{\mathbb{Z}},H^1(I_v,\hat{T}))^{D}}{S_v} \ar[r]  & \prod_{v\notin S} H^1(K_v,T) \ar[r] & \prod_{v\notin S} Ext^1_{\hat{\mathbb{Z}}}(\hat{T}^{I_v},\mathbb{Z}) \ar[r] & 0  }
\]
By the Snake Lemma, we obtain
\begin{equation}\label{value_galois_seq11}
 0 \to \mathbb{III}^1_{S}(T) \to \ker(\Theta^1_S) \to \prod_{v\notin S}
 \frac{H^0(\hat{\mathbb{Z}},H^1(I_v,\hat{T}))^{D}}{S_v} 
\to \mathrm{cok}(\Delta^1_{S}) \to \mathrm{cok}(\Theta^1_S) \to 0.
\end{equation}
 As $\mathrm{cok}(\Delta^1_{S})$ and $\prod_{v\notin S}S_v$ are finite, so are $\mathrm{cok}(\Theta^1_S)$ and
 $\prod_{v\notin S}H^0(\hat{\mathbb{Z}},H^1(I_v,\hat{T}))$. Therefore from ($\ref{value_galois_seq11}$) and Lemma $\ref{lemma_Sha_infty}$,
\begin{equation}\label{value_galois_seq13}
[\ker(\Theta^1_S)]=\frac{[\mathrm{cok}(\Theta^1_S)][\mathbb{III}^1(T)][\mathbb{III}^2(T)]\prod_{v\in S}[H^1(K_v,T)]}{[H^1(K,\hat{T})]}
\prod_{v\notin S}\frac{[H^0(\hat{\mathbb{Z}},H^1(I_v,\hat{T}))^{D}]}{[S_v]}.
\end{equation}
Putting together ($\ref{value_galois_seq8}$) and ($\ref{value_galois_seq13}$), we obtain 
\begin{equation}\label{value_galois_seq14}
h_{T,S}=\frac{[Ext_U^1(j_{*}\hat{T},\mathbb{G}_m)][H^1(K,\hat{T})]}
{[\mathrm{cok}(\Theta^1_S)][\mathbb{III}^1(T)][\mathbb{III}^2(T)]\prod_{v\in S}[H^1(K_v,T)]\prod_{v\notin S}[H^0(\hat{\mathbb{Z}},H^1(I_v,\hat{T}))]}.
\end{equation}
The theorem will follow from the next lemma.
\end{proof}
\begin{lemma}Notations as in Theorem $\ref{thm_hTS}$. Then
\begin{equation}
[\mathbb{III}^2(T)][\mathrm{cok}(\Theta^1_S)]=[\mathrm{ker}\Psi^1(j_{*}\hat{T})].
\end{equation}
\end{lemma}
\begin{proof}
From ($\ref{seq_compact_coh}$), we have an exact sequence 
\begin{equation}\label{seq_2nd_lemma_1}
H^0(K,\hat{T}) \to \prod_{v\in S}H^0(K_v,\hat{T}) \to H^1_c(U,j_{*}\hat{T})\to \mathrm{ker}\Psi^1(j_{*}\hat{T}) \to 0.
\end{equation}
Let $R_1$ be the cokernel of the map $ H^0(K,\hat{T})\to \prod_{v\in S}H^0(K_v,\hat{T})$. Then $R_1^D$ can be identified with the kernel of the map
\[ \prod_{v\in S}H^2(K_v,{T}) \to H^0(K,\hat{T})^D \] 
 from ($\ref{Poitou_Tate}$). 
 As $Ext^2_U(j_{*}\hat{T},\mathbb{G}_m)\simeq H^1_c(U,j_{*}\hat{T})^D$, taking dual of ($\ref{seq_2nd_lemma_1}$) yields
\[ 0 \to (\mathrm{ker}\Psi^1(j_{*}\hat{T}))^D \to Ext^2_U(j_{*}\hat{T},\mathbb{G}_m) \to R_1^D \to 0.\]
From ($\ref{seq_main_lemma_1}$) for $n=2$, we have the commutative diagram
\begin{equation}
\xymatrix{
0 \ar[r] & (\mathrm{ker}\Psi^1(j_{*}\hat{T}))^D \ar[r] \ar[d] & Ext_U^2(j_{*}\hat{T},\mathbb{G}_m) \ar[r]\ar@{->>}[d] & R_1^D \ar[r] \ar[d]^{\simeq} & 0 \\
0 \ar[r] & \mathbb{III}^2(T) \ar[r] & \mathbb{III}_S^2(T) \ar[r] & R_1^D \ar[r] & 0.
}
\end{equation}
The lemma follows from ($\ref{value_galois_seq6}$) and diagram chasing. As a result, Theorem $\ref{thm_hTS}$ is proved.   
\end{proof}
\begin{cor}\label{cor_Hom_UT}
$T(O_{K,S})\simeq Hom_U(j_{*}\hat{T},\mathbb{G}_m)$.
\end{cor}
\section{Class Number Formulas for Norm Tori and Their Duals}
Let $L/K$ be a Galois extension of number fields with Galois group $G$. 
For each prime $v$ of $K$, choose a prime $w$ of $L$ lying over $v$. If $v$ is finite, let $D_w$ and $I_w$ be the decomposition group and the inertia group of $w$. Let $e_v(L/K)$ be the ramification index of $v$ in $L$. For any finite Galois extension $L/K$, let $L'/K$ be the maximal abelian subextension of $L/K$. Let $S$ be a finite set of places of $K$ containing $S_{\infty}$ and $S'$ be the set of places of $L$ which lie over a place in $S$. Let $\pi:Spec(L)\to Spec(K)$ and $\pi':Spec(O_{L,S'})\to Spec(O_{K,S})$ be the induced maps. If $G$ is a finite group and $M$ is a discrete $G$-module, we write $H^n_T(G,M)$ for the Tate cohomology. 

  Let $T=R_{L/K}^{(1)}(\mathbb{G}_m)$ be the norm torus corresponding to $L/K$ and $T'=R_{L/K}(\mathbb{G}_m)/\mathbb{G}_m$ be the dual torus of $T$. In this section, we shall compute $h_{T,S}$ and $h_{T',S}$ using $\eqref{eqn_main}$.

\begin{lemma}\label{lemma_norm_class_number}
Let $T=R_{L/K}^{(1)}(\mathbb{G}_m)$. Then
\begin{enumerate}
\item $H^0(K,\hat{T})=0$ and $[H^1(K,\hat{T})]=[L':K]$.
\item For any place $v$, $[H^1(K_v,{T})]= [H^1(K_v,\hat{T})]= [L_w':K_v]$.
\item For $v\notin S_{\infty}$, if $w$ is a place of $L$ lying above $v$ then $[H^0(\hat{\mathbb{Z}},H^1(I_v,\hat{T}))]= e_v(L'/K)$.
\item Let $I_L$ be the group of ideles of $L$. Then $\mathbb{III}^1(T)\simeq \ker(H^0_T(G,L^{*})\to H^0_T(G,I_L))$ 
\end{enumerate}
\end{lemma}
\begin{proof}
\begin{enumerate}
\item We have the exact sequence
\begin{equation}\label{lemma_norm_class_number_eqn1}
0 \to \mathbb{Z} \to \mathbb{Z}[G] \to \hat{T} \to 0. 
\end{equation}
From ($\ref{lemma_norm_class_number_eqn1}$) and the fact that $H^n(G,\mathbb{Z}[G])=0$ for $n>0$, we deduce that
\begin{eqnarray*}
H^0(K,\hat{T})&\simeq & H^0(G,\hat{T})\simeq H^1(G,\mathbb{Z})=0, \\
H^1(K,\hat{T})&\simeq & H^1(G,\hat{T})\simeq H^2(G,\mathbb{Z}) \simeq H^1(G,\mathbb{Q}/\mathbb{Z})\simeq Hom_{\mathbb{Z}}(G^{ab},\mathbb{Q}/\mathbb{Z}).
\end{eqnarray*}
\item From Tate's local duality, $[H^1(K_v,{T})]= [H^1(K_v,\hat{T})]=[H^1(D_w,\hat{T})]$.
Since $\mathbb{Z}[G]$ is also an induced $D_w$-module, part 2) can be obtained from $\eqref{lemma_norm_class_number_eqn1}$ as in part 1).
\item We have $H^0(\hat{\mathbb{Z}},H^1(I_v,\hat{T}))\simeq H^0(D_w/I_w,H^1(I_w,\hat{T}))$. From sequence ($\ref{rel_class_number_eqn2}$) and the fact that $\mathbb{Z}[G]$ is an induced $I_w$-module, $H^1(I_w,\hat{T})\simeq H^2(I_w,\mathbb{Z})$. Consider the spectral sequence 
\[ E_2^{m,n}=H^m(D_w/I_w,H^n(I_w,\mathbb{Z}))\Rightarrow E^{m+n}=H^{m+n}(D_w,\mathbb{Z}).\]
Note that $E_2^{m,1}=0$ for all $m$ and $E_2^{m,0}=0$ for $m$ odd. Therefore, we obtain the exact sequence
\[ 0 \to (D_w/I_w)^D \to D_w^D \to H^0(D_w/I_w,H^2(I_w,\mathbb{Z})) \to 0 .\]
Hence, $H^0(D_w/I_w,H^2(I_w,\mathbb{Z}))\simeq I_w^D$. Therefore, $[H^0(\hat{\mathbb{Z}},H^1(I_v,\hat{T}))]= e_v(L'/K)$.
\item This is proved in $\cite[\mbox{page 307}]{PR92}$. 
\end{enumerate}
\end{proof}
\begin{thm}[\cite{Ono87},\cite{Mor91}]\label{thm_rel_class_number}
 Let $T=R_{L/K}^{(1)}(\mathbb{G}_m)$. Then 
\begin{equation}
h_{T,S}= \frac{h_{L,S'}[L':K][H^0_T(G,O_{L,S'}^{*})]}{h_{K,S}[\ker(H^0_T(G,L^{*})\to H^0_T(G,I_L))]\prod_{v\in S}[L_w':K_v] \prod_{v\notin S}e_v(L'/K)}.
\end{equation}
\end{thm}
\begin{proof}
We have the exact sequence of $G_K$-modules
\begin{equation}\label{rel_class_number_eqn1}
0 \to \mathbb{Z} \to \pi_{*}\mathbb{Z} \to \hat{T} \to 0.
\end{equation}
Since $R^1j_{*}\mathbb{Z}=0$, ($\ref{rel_class_number_eqn1}$) induces the exact sequence of \'etale sheaves on $U=Spec(O_{K,S})$
\[ 0 \to \mathbb{Z} \to \pi'_{*}\mathbb{Z} \to j_{*}\hat{T} \to 0. \]
The long exact sequence of Ext-groups yields  
\begin{equation}\label{rel_class_number_eqn2}
0 \to  H^0_T(G,O_L^{*}) \to 
Ext^1_{U}(j_{*}\hat{T},\mathbb{G}_m) \to Pic(O_{L,S'})\to Pic(O_{K,S}) \to R \to 0.
\end{equation}
By Artin-Verdier Duality, $R^D$ is the kernel of the map
$   H^2_{c}(U,\mathbb{Z})\to 
H^2_{c}(U,\pi'_{*}\mathbb{Z})$ .
From $\cite[\mathrm{II.2.11}]{Mil06}$ and the fact that $H^1_{et}(U,\mathbb{Z})=0$ and $H^1(K_v,\mathbb{Z})=0$, there is a commutative diagram
\begin{equation*}
\xymatrix{
         &  &  0 \ar[d] & 0 \ar[d] \\
         &     & H^1_{et}(U,j_{*}\hat{T}) \ar[r]^{\Psi^1(j_{*}\hat{T})} \ar[d] & \prod_{v\in S}H^1(K_v,\hat{T}) \ar[d] \\
0 \ar[r] & H^2_c(U,\mathbb{Z}) \ar[r] \ar[d] & H^2_{et}(U,\mathbb{Z}) \ar[r] \ar[d]
& \prod_{v\in S}H^2(K_v,\mathbb{Z}) \ar[d] \\
0\ar[r]  & H^2_c(U,\pi'_{*}\mathbb{Z})\ar[r]  & H^2_{et}(U,\pi'_{*}\mathbb{Z}) \ar[r] 
& \prod_{v\in S}H^2(K_v,\pi'_{*}\mathbb{Z})  
}
\end{equation*}
By diagram chasing, $R^D \simeq \mathrm{ker}\Psi^1(j_{*}\hat{T})$. By ($\ref{rel_class_number_eqn2}$), 
\begin{equation}\label{rel_class_number_eqn3}
\frac{[Ext^1_{U}(j_{*}\hat{T},\mathbb{G}_m)]}{[\mathrm{ker}\Psi^1(j_{*}\hat{T})]}
=\frac{h_{L,S'}[H^0_T(G,O_{L,S'}^{*})]}{h_{K,S}}.
\end{equation} 
Now ($\ref{thm_rel_class_number}$) follows from ($\ref{rel_class_number_eqn3}$), Theorem $\ref{thm_hTS}$ and Lemma $\ref{lemma_norm_class_number}$. 
\end{proof}
\begin{rmk}
We note a simple but useful fact : if $Pic(O_K)=0$ then $\ker\Psi^1(j_{*}\hat{T})=0$.
\end{rmk}
\begin{lemma}\label{lemma_proj_class_number}
Let $T'=R_{L/K}(\mathbb{G}_m)/\mathbb{G}_m$. Then
\begin{enumerate}
\item $H^0(K,\hat{T}')=0$ and $H^1(K,\hat{T}')\simeq \mathbb{Z}/[G]\mathbb{Z}$.
\item For $v\notin S_{\infty}$, 
$[H^0(\hat{\mathbb{Z}},H^1(I_v,\hat{T}'))] = e_v(L/K)$.
\item $\mathbb{III}^1(T')=0$.
\end{enumerate}
\end{lemma}
\begin{proof}
The character group of $T'$ satisfies the following exact sequence
\begin{equation}\label{seq_lemma_proj_class_number_1}
0 \to \hat{T}' \to \mathbb{Z}[G] \xrightarrow{\epsilon} \mathbb{Z}\to 0.
\end{equation}
Consider the long exact sequence of $\eqref{seq_proj_class_number_1}$
\[ 0 \to H^0(G,\hat{T}') \to H^0(G,\mathbb{Z}[G]) \xrightarrow{\epsilon} H^0(G,\mathbb{Z}) \to H^1(G,\hat{T}') \to 0.\]
The augmentation map $\epsilon$ in this case is just the multiplication-by-$[G]$ map. Hence, part 1) follows. Part 3) is proved in $\cite[\mbox{page 685}]{Kat91}$. Let us prove part 2). Let $w$ be a prime of $L$ dividing $v$.
\[ H^0(\hat{\mathbb{Z}},H^1(I_v,\hat{T}'))\simeq H^0(D_w/I_w,H^1(I_w,\hat{T}')).\]  
By ($\ref{seq_lemma_proj_class_number_1}$), $H^1(I_w,\hat{T}')\simeq H^0_T(I_w,\mathbb{Z})$. As $D_w/I_w$ acts trivially on $H^0_T(I_w,\mathbb{Z})$,
\[ [H^0(D_w/I_w,H^1(I_w,\hat{T}'))]=[H^0_T(I_w,\mathbb{Z})]=[I_w]=e_v(L/K).\]
\end{proof}
\begin{thm}[$\cite{Kat91}$]\label{thm_proj_class_number}
Let $T'=R_{L/K}(\mathbb{G}_m)/\mathbb{G}_m$. Then
\begin{equation}
h_{T',S}= \frac{h_{L,S'}[H^1(G,O_{L,S'}^{*})]}{h_{K,S} \prod_{v\notin S}e_v(L/K)}.
\end{equation}
\end{thm}
\begin{proof}
We have the following exact sequence of \'etale sheaves on $Spec(K)$
\begin{equation}\label{seq_proj_class_number_1}
0 \to \hat{T}' \to \pi_{*}\mathbb{Z} \xrightarrow{\epsilon} \mathbb{Z}\to 0.
\end{equation}
As $R^1j_{*}(\pi_{*}\mathbb{Z})=0$, ($\ref{seq_proj_class_number_1}$) induces
\begin{equation}\label{seq_proj_class_number_2}
0 \to j_{*}\hat{T}' \to \pi'_{*}\mathbb{Z} \xrightarrow{\epsilon} \mathbb{Z} \to R^1j_{*}\hat{T}' \to 0.
\end{equation}
To ease notation, let $\mathcal{R}=R^1j_{*}\hat{T}'$.  We split ($\ref{seq_proj_class_number_2}$) into 
\begin{equation}\label{seq_proj_class_number_3a}
0 \to j_{*}\hat{T}' \to \pi'_{*}\mathbb{Z} \xrightarrow{\alpha} \mathcal{Q} \to 0.
\end{equation}
\begin{equation}\label{seq_proj_class_number_3b}
0 \to \mathcal{Q} \xrightarrow{\beta} \mathbb{Z} \to \mathcal{R} \to 0.
\end{equation} 
Let $\beta_c^n$ be the map $H^n_{c}(X,\mathcal{Q})\to H^n_{c}(X,\mathbb{Z})$.
From the long exact sequences of $Ext$-groups of ($\ref{seq_proj_class_number_3b}$), we obtain $Hom_U(\mathcal{Q},\mathbb{G}_m) \simeq O_{K,S}^{*}$ and
\begin{eqnarray}
[Ext^1_U(\mathcal{Q},\mathbb{G}_m)]
= \frac{h_{K,S} [Ext^2_U(\mathcal{R},\mathbb{G}_m)]}{[\mathrm{cok}(\beta_c^1)]}.
\end{eqnarray}
Let $\alpha_c^n$ be the map $H^n_{c}(X,\pi'_{*}\mathbb{Z})\to H^n_{c}(X,\mathcal{Q})$.
Similar argument applied to sequence ($\ref{seq_proj_class_number_3a}$) yields 
\begin{eqnarray*}
[Ext^1_U(j_{*}\hat{T}',\mathbb{G}_m)]
&=& \frac{h_{L,S'}[\mathrm{cok}(\alpha^1_c)][S]}{[Ext^1_U(\mathcal{Q},\mathbb{G}_m)]}
=\frac{h_{L,S'}[S][\mathrm{cok}(\alpha^1_c)][\mathrm{cok}(\beta_c^1)]}
{h_{K,S}[Ext^2_U(\mathcal{R},\mathbb{G}_m)]}
\end{eqnarray*}
where $S$ satisfies the exact sequence
\begin{equation}\label{seq_proj_class_number_4}
0 \to O_K^{*} \to O_L^{*} \to Hom_U(j_{*}\hat{T}',\mathbb{G}_m) \to S \to 0.
\end{equation}
We claim that $S\simeq H^1(G,O_{L,S'}^{*})$. Indeed,
from Corollary $\ref{cor_Hom_UT}$, $Hom_U(j_{*}\hat{T}',\mathbb{G}_m)\simeq T'(O_{K,S})$. Therefore, ($\ref{seq_proj_class_number_4}$) can be identified with the sequence
\begin{equation}\label{seq_proj_class_number_6}
0 \to \mathbb{G}_m(O_{K,S}) \to R_{L/K}(\mathbb{G}_m)(O_{K,S}) \to T'(O_{K,S}) \to S \to 0 
\end{equation}
which is part of the long exact sequence of cohomology associated with 
\begin{equation}\label{seq_proj_class_number_7}
0 \to \mathbb{G}_m(O_{L,S'}) \to R_{L/K}(\mathbb{G}_m)(O_{L,S'}) \to T'(O_{L,S'}) \to 0. 
\end{equation}
 Note that $R_{L/K}(\mathbb{G}_m)(O_{L,S'})$ is an induced $G$-module thus $H^1(G,R_{L/K}(\mathbb{G}_m)(O_{L,S'}))=0$. Consider the long exact sequence of cohomology associated with ($\ref{seq_proj_class_number_7}$) and compare with ($\ref{seq_proj_class_number_6}$), we obtain $S\simeq H^1(G,O_{L,S'}^{*})$. 
 From Theorem $\ref{thm_hTS}$ and Lemma $\ref{lemma_proj_class_number}$, we have 
 \[h_{T',S}= \frac{h_{L,S'}[H^1(G,O_{L,S'}^{*})][L:K][\mathrm{cok}(\alpha^1_c)][\mathrm{cok}(\beta^1_c)]}
 {h_{K,S} \prod_{v\notin S}e_v(L/K)
\prod_{v\in S}[H^1(K_v,\hat{T}')][\mathrm{ker}\Psi^1(j_{*}\hat{T}')][Ext^2_U(\mathcal{R},\mathbb{G}_m)]}.
 \]
Theorem $\ref{thm_proj_class_number}$ will follow from the next lemma.
 \end{proof}
 \begin{lemma}\label{lemma_proj_class_number_2}
 Notations as in Theorem $\ref{thm_proj_class_number}$. Then
 \begin{equation}
  \frac{[\mathrm{cok}(\alpha^1_c)][\mathrm{cok}(\beta^1_c)][L:K]}
{[Ext^2_U(\mathcal{R},\mathbb{G}_m)][\mathrm{ker}\Psi^1(j_{*}\hat{T}')]}=\prod_{v\in S}[H^1(K_v,\hat{T}')].
 \end{equation}
 \end{lemma}
 \begin{proof}
 Consider the following commutative diagram 
  \begin{equation*}
\xymatrix{
0 \ar[r] & H^0_c(U,\mathcal{Q}) \ar[r] \ar[d]^{\beta^0_c} & H^0_{et}(U,\mathcal{Q}) \ar[r] \ar[d]^{\beta^0_{et}} 
& \underset{v\in S}{\prod}H^0(K_v,\mathcal{Q}_v)\ar[r] \ar[d]^{\beta^0_{S}}  & 
H^1_c(U,\mathcal{Q}) \ar[r] \ar[d]^{\beta^1_c}  & H^1_{et}(U,\mathcal{Q}) \ar[r] \ar[d]^{\beta^1_{et}} & 0\\
0\ar[r]  & H^0_c(U,\mathbb{Z})\ar[r]  & H^0_{et}(U,\mathbb{Z}) \ar[r] 
& \underset{v\in S}{\prod}H^0(K_v,\mathbb{Z}) \ar[r] & H^1_c(U,\mathbb{Z}) \ar[r] & 0 \ar[r] & 0
}
\end{equation*}
As $\mathcal{R}$ is negligible, $\mathcal{R}_v=0$, $\mathcal{Q}_v\simeq \mathbb{Z}$ and $\beta^0_S$ is an isomorphism. Hence $\mathrm{cok}(\beta^1_c)=0$. Next we consider the diagram.
\begin{equation*}
\xymatrix{
0 \ar[r] & H^0_c(U,\pi'_{*}\mathbb{Z}) \ar[r] \ar[d]^{\alpha^0_c} & H^0_{et}(U,\pi'_{*}\mathbb{Z}) \ar[r] \ar[d]^{\alpha^0_{et}} 
& \underset{v\in S}{\prod}H^0(K_v,\pi'_{*}\mathbb{Z})\ar[r] \ar[d]^{\alpha^0_{S}}  & 
H^1_c(U,\pi'_{*}\mathbb{Z}) \ar[r] \ar[d]^{\alpha^1_c}  & 0 \ar[d] \ar[r] & 0\\
0 \ar[r] & H^0_c(U,\mathcal{Q}) \ar[r]  & H^0_{et}(U,\mathcal{Q}) \ar[d] \ar[r]  
& \underset{v\in S}{\prod}H^0(K_v,\mathcal{Q}_v)\ar[r]  \ar[d] & 
H^1_c(U,\mathcal{Q}) \ar[r]  & H^1_{et}(U,\mathcal{Q}) \ar[r] & 0 \\
         &  &   H^1_{et}(U,j_{*}\hat{T}') \ar[r]^{\Psi^1(j_{*}\hat{T}')} \ar[d]& \underset{v\in S}{\prod}
         H^1(K_v,\hat{T}')  \ar[d] \\
         &  &   0  & 0
}
\end{equation*}
By diagram chasing, we have 
\begin{eqnarray}
[\mathrm{cok}(\alpha^1_c)]=[H^1_{et}(U,\mathcal{Q})][\mathrm{cok}\Psi^1(j_{*}\hat{T}')] 
= \frac{[H^1_{et}(U,\mathcal{Q})][\mathrm{ker}\Psi^1(j_{*}\hat{T}')]{\prod}_{v\in S} [H^1(K_v,\hat{T}')]}{[H^1_{et}(U,j_{*}\hat{T}')]}.
\end{eqnarray}
It is not hard to show that $[\mathrm{cok}(\alpha^0_c)]=[H^1_{et}(U,j_{*}\hat{T}')]$, 
$[\mathrm{cok}(\beta^0_c)]=[H^0_{et}(U,\mathcal{R})]/[H^1_{et}(U,\mathcal{Q})]$. Let $\epsilon^n_c : H^n_{c}(U,\pi'_{*}\mathbb{Z})\to H^n_{c}(U,\mathbb{Z})$ be the composition $\beta^n_c\circ\alpha^n_c$.
Note that $\epsilon^0_{et}$ can be identified with the augmentation map $H^0(G,\mathbb{Z}[G]) \to H^0(G,\mathbb{Z})$. Hence, $\mathrm{cok}(\epsilon^0_{et})\simeq  \mathbb{Z}/[L:K]\mathbb{Z}$. As $\beta^0_c$ is injective, 
$[\mathrm{cok}(\alpha^0_c)][\mathrm{cok}(\beta^0_c)]=[\mathrm{cok}(\epsilon^0_c)]$. Therefore,
\begin{equation}
[L:K]=\frac{[H^1_{et}(U,j_{*}\hat{T}')][H^0_{et}(U,\mathcal{R})]}{[H^1_{et}(U,\mathcal{Q})]}
=\frac{[H^1_{et}(U,j_{*}\hat{T}')][Ext^2_{U}(\mathcal{R},\mathbb{G}_m)]}{[H^1_{et}(U,\mathcal{Q})]}
\end{equation}
That completes the proofs of the lemma and Theorem $\ref{thm_proj_class_number}$.
 \end{proof}
 \begin{cor}[\cite{Neu99}\mbox{VI.3.5}]\label{cyclic_rel_class_number}
Suppose $L/K$ is a cyclic extension. Then the Herbrand quotient of $O_{L,S'}^{*}$ is given by
\begin{equation*}
h(G,O^{*}_{L,S'})= \frac{\prod_{v\in S}[L_w:K_v]}{[L:K]}.
\end{equation*}
\end{cor}
\begin{proof}
Since $L/K$ is cyclic, by Hasse's theorem $\ker(H^0_T(G,L^{*})\to H^0_T(G,I_L))=0$. In addition, $T\simeq T'$. Thus, the corollary follows by comparing the two formulas in Theorems $\ref{thm_rel_class_number}$ and $\ref{thm_proj_class_number}$.
\end{proof}
 \section{A Relative Class Number Formula for Isogenous Tori}
 Let $\lambda: T \to T'$ be an isogeny between two algebraic tori defined over a number field $K$. In other words, we have  exact sequences 
\begin{equation}\label{seq_shyr_1}
0 \to T'' \to T \to T' \to 0 \quad \& \quad  0 \to \hat{T}' \to \hat{T} \to \hat{T}'' \to 0
\end{equation}
where $T''$ is a finite algebraic group over $K$ and $\hat{T}''$ is a finite discrete $G_K$-module. Then $\lambda$ induces the following maps all of which have finite kernels and cokernels.
\begin{itemize}
\item $\lambda(O_{K,S}): T(O_{K,S}) \to T'(O_{K,S})$.
\item $\hat{\lambda}(H^n_{et}) : H^n_{et}(X,j_{*}\hat{T}') \to H^n_{et}(X,j_{*}\hat{T})$ and
$\hat{\lambda}(H^n_{c}) : H^n_{c}(X,j_{*}\hat{T}') \to H^n_{c}(X,j_{*}\hat{T})$.
\item Let $\alpha(T)=\mathrm{ker}\Psi^1(j_{*}\hat{T})$ and similarly for $\alpha(T')$. Then there is  a map $\hat{\lambda}(\alpha):\alpha(T')\to \alpha(T)$.
\item For finite prime $v$, let $c_v(T)=H^0(\hat{\mathbb{Z}},H^1(I_v,\hat{T}))^{D}$ and similarly for $c_v(T')$. We have
$\lambda(O_v) : T_v^c \to {T'}_v^c$ and $\hat{\lambda}(c_v):c_v(T)\to c_v(T')$.
\item For any prime $v$, $\lambda(H^n(K_v)) : H^n(K_v,T) \to H^n(K_v,T')$ and $\hat{\lambda}(H^n(K_v)) : H^n(K_v,\hat{T}') \to H^n(K_v,\hat{T})$.
\item For infinite prime $v$, $\lambda(H^n_T(K_v)) : H^n_T(K_v,T) \to H^n_T(K_v,T')$ and $\hat{\lambda}(H^n_T(K_v)) : H^n_T(K_v,\hat{T}') \to H^n_T(K_v,\hat{T})$.
\end{itemize}
Let $\alpha$ be a group homomorphism with finite kernel and cokernel. We define  $q(\alpha):=[\mathrm{cok}(\alpha)]/[\ker(\alpha)]$. Then $q(\alpha)$ is multiplicative with respect to exact sequences  $\cite[\mbox{0.3.1}]{Ono61}$. The purpose of this section is to generalize a formula of Shyr for $h_{T,S}/h_{T',S}$ using $\eqref{eqn_main1}$.
\begin{lemma}\label{lemma_qH0/H1_a}
\[
\frac{q(\hat{\lambda}(H^0(K_v))}{q(\hat{\lambda}(H^{-1}(K_v))}
=\left\{ \begin{array}{cl}
{[\mathrm{cok}(\hat{\lambda}(H^{0}(K_v))]} & \mbox{$v\notin S_{\infty}$, }\\
\frac{[\ker(\hat{\lambda}(H^{-1}(K_v))][\mathrm{cok}(\hat{\lambda}(H^{0}(K_v))]}{[H^0(K_v,\hat{T}'')]} & \mbox{$v\in S_{\infty}$. }
\end{array}
\right.
\]
\end{lemma}
\begin{proof}
If $v\notin S_{\infty}$ then this is clear.
If $v\in S_{\infty}$ then this follows from diagram chasing and the fact that
$[H^{-1}_T(K_v,\hat{T}'')]=[H^{0}_T(K_v,\hat{T}'')]$.
\end{proof}
 \begin{lemma}\label{lemma_qH0/H1}For any prime $v$, 
 \[ \frac{q(\lambda(H^1(K_v))}{q(\lambda(H^0(K_v))}
=\frac{|[\hat{T}'']|_v}{[\mathrm{cok}(\hat{\lambda}(H^0(K_v)))]} 
 .\]
 \end{lemma}
 \begin{proof}
 From the exact sequence
 \[ 0 \to H^0(K_v,T'')\to H^0(K_v,T)\to H^0(K_v,T') \to H^1(K_v,T'')\to 
 \ker({\lambda}(H^{1}(K_v)) \to 0\]
 and the facts that $[\ker({\lambda}(H^{1}(K_v))] =[\mathrm{cok}(\hat{\lambda}(H^{1}(K_v))]$ and $[H^n(K_v,T'')]= [H^{2-n}(K_v,\hat{T}'')]$, we obtain
 \[ \frac{q(\lambda(H^1(K_v))}{q(\lambda(H^0(K_v))}=
\frac{[H^2(K_v,\hat{T}'')][\ker(\hat{\lambda}(H^{1}(K_v))]}{[H^1(K_v,\hat{T}'')]}
=\frac{|[\hat{T''}]|_v[\ker(\hat{\lambda}(H^{1}(K_v))]}{[H^0(K_v,\hat{T}'')]}=\frac{|[\hat{T}'']|_v}{[\mathrm{cok}(\hat{\lambda}(H^0(K_v)))]}.
\]
where the second equality follows from $\cite[\mbox{I.2.8}]{Mil06}$.
 \end{proof}
 
 \begin{lemma}\label{lemma_shyr_local}
Let $v$ be a finite prime of $K$. Then
\begin{equation}\label{seq_shyr_local}
q(\lambda(c_v))=q(\lambda(O_v))|[\hat{T}'']|_v .
\end{equation} 
\end{lemma}
\begin{proof}
From Corollary $\ref{local_duality_examp}$, we have the following commutative diagram 
\begin{equation}
\xymatrix{
{T_v^c}\ar[r] \ar[d]^{\lambda(O_v)} & {T(K_v)}\ar[r] \ar[d]^{\lambda(H^0(K_v))} & Hom_{\hat{\mathbb{Z}}}(\hat{T}^{I_v},\mathbb{Z}) \ar[r] \ar[d]^{\lambda(Hom)}
& H^0(\hat{\mathbb{Z}},H^1(I_v,\hat{T}))^D \ar[r] \ar[d]^{\hat{\lambda}(c_v)} & H^1(K_v,T)\ar[r]\ar[d]^{\lambda(H^1(K_v))} & Ext^{1}_{\hat{\mathbb{Z}}}(\hat{T}^{I_v},\mathbb{Z}) \ar[d]^{\lambda(Ext)} \\
{{T'}_v^c}\ar[r] & {T'(K_v)}\ar[r] & Hom_{\hat{\mathbb{Z}}}(\hat{T}'^{I_v},\mathbb{Z}) \ar[r]
& H^0(\hat{\mathbb{Z}},H^1(I_v,\hat{T}'))^D \ar[r] & H^1(K_v,T')\ar[r] & Ext^{1}_{\hat{\mathbb{Z}}}(\hat{T}'^{I_v},\mathbb{Z})}
\end{equation}
Note that all the vertical maps have finite kernels and cokernels. We have 
\begin{equation}\label{seq_shyr_local_0}
\frac{q(\hat{\lambda}(c_v))}{q(\lambda(O_v))}=\frac{q(\lambda(Hom))q(\lambda(H^1(K_v)))}{q(\lambda(Ext))q(\lambda(H^0(K_v)))}. 
\end{equation}
From ($\ref{seq_shyr_1}$), we have an exact sequence of $\hat{\mathbb{Z}}$-modules
\begin{equation*}
0 \to \hat{T}'^{I_v} \to \hat{T}^{I_v} \to P \to 0
\end{equation*}
where $P$ is finite.
As $P$ is finite, $[ Ext^n_{\hat{\mathbb{Z}}}(P,\mathbb{Z})]=[H^{2-n}(\hat{\mathbb{Z}},P)]$. In particular, $[ Ext^1_{\hat{\mathbb{Z}}}(P,\mathbb{Z})]=[Ext^2_{\hat{\mathbb{Z}}}(P,\mathbb{Z})]$.  Moreover, 
$Ext^2_{\hat{\mathbb{Z}}}(\hat{T}^{I_v},\mathbb{Z})\simeq H^0(\hat{\mathbb{Z}},\hat{T}^{I_v})^D \simeq H^0(K_v,\hat{T})^D$. Thus,
$ [\ker(\lambda(Ext^2))]=[\mathrm{cok}(\hat{\lambda}(H^0(K_v)))]$. By direct calculation,
\begin{equation}\label{seq_shyr_local_5}
\frac{q(\lambda(Hom))}{q(\lambda(Ext))}=\frac{[Ext^1_{\hat{\mathbb{Z}}}(P,\mathbb{Z})][\ker(\lambda(Ext^2))]}
{[Ext^2_{\hat{\mathbb{Z}}}(P,\mathbb{Z})]}=[\mathrm{cok}(\hat{\lambda}(H^0(K_v)))].
\end{equation}
Now put together ($\ref{seq_shyr_local_0}$),($\ref{seq_shyr_local_5}$) and Lemma $\ref{lemma_qH0/H1}$, we obtain ($\ref{seq_shyr_local}$). 
\end{proof}

 \begin{thm}[\cite{Shy77}]\label{thm_shyr}
  Let $\lambda: T \to T'$ be an isogeny between two algebraic tori defined over a number field $K$. Let $\tau(T)$ and $\tau(T')$ be the Tamagawa numbers of $T$ and $T'$ respectively. Then 
  \[ \frac{h_{T,S}}{h_{T',S}}=
  \frac{\tau(T)\prod_{v\in S}q(\lambda(H^0(K_v)))\prod_{v\notin S}q(\lambda(O_v))}{\tau(T')q(\lambda(O_{K,S}))q(\hat{\lambda}(H^0_{et}))}.\]
 \end{thm}
\begin{proof}
By Theorem $\ref{thm_hTS}$ and the fact that $\tau(T)=[H^1(K,\hat{T})]/[\mathbb{III}^1(T)]$ (see \cite{Ono63}),
\begin{equation}\label{seq_shyr_0}
\frac{h_{T,S}}{h_{T',S}}=
\frac{\tau(T)}{\tau(T')}\frac{[Ext^1_U(j_{*}\hat{T},\mathbb{G}_m)]}
{[Ext^1_U(j_{*}\hat{T}',\mathbb{G}_m)]}
\frac{\prod_{v\notin S}{q(\hat{\lambda}(c_v))}
\prod_{v\in S}{q({\lambda}(H^1(K_v)))}}
{q(\hat{\lambda}(\alpha))}.
\end{equation}
Combining Lemmas $\ref{lemma_qH0/H1}$ and $\ref{lemma_shyr_local}$, 
\begin{equation}\label{seq_shyr_9}
\prod_{v\notin S}{q(\hat{\lambda}(c_v))}\prod_{v\in S}{q({\lambda}(H^1(K_v)))}
=\prod_{v\notin S}{q({\lambda}(O_v))}
\prod_{v\in S}\frac{q({\lambda}(H^0(K_v)))}{{[\mathrm{cok}(\hat{\lambda}(H^0(K_v)))]}}.
\end{equation}
 Sequence ($\ref{seq_shyr_1}$) induces an exact sequence of sheaves 
\[ 0 \to j_{*}\hat{T}' \to j_{*}\hat{T} \to j_{*}\hat{T}'' \to \mathcal{R} \to 0\]
where $\mathcal{R}$ is negligible as it is a subsheaf of $R^1j_{*}\hat{T}'$. We can split this sequence into 
\begin{equation}\label{seq_shyr_2a}
 0 \to j_{*}\hat{T}' \to j_{*}\hat{T} \to \mathcal{Q} \to 0, 
\end{equation}
\begin{equation}\label{seq_shyr_2b}
0\to \mathcal{Q} \to j_{*}\hat{T}'' \to \mathcal{R} \to 0
\end{equation}
where $\mathcal{Q}$ is a constructible sheaf. The long exact sequence of $\mathrm{Ext}$-groups associated to ($\ref{seq_shyr_2a}$) can be split into
\begin{equation}\label{seq_shyr_3a}
0 \to Hom_U(\mathcal{Q},\mathbb{G}_m) \to  Hom_U(j_{*}\hat{T},\mathbb{G}_m) 
\xrightarrow{\lambda(O_{K,S})}  
Hom_U(j_{*}\hat{T}',\mathbb{G}_m) \to S_1 \to 0,
\end{equation}
\begin{equation}\label{seq_shyr_3b}
0 \to S_1 \to Ext^1_U(\mathcal{Q},\mathbb{G}_m) \to  Ext^1_U(j_{*}\hat{T},\mathbb{G}_m) \to  
  Ext^1_U(j_{*}\hat{T}',\mathbb{G}_m)\to Ext^2_U(\mathcal{Q},\mathbb{G}_m) \to S_2 \to 0.
\end{equation}
The Artin-Verdier duality implies $S_2^D\simeq \mathrm{cok}(\hat{\lambda}(H^1_c))$.
As $Hom_U(j_{*}\hat{T},\mathbb{G}_m)\simeq T(O_{K,S})$, ($\ref{seq_shyr_3a}$) yields
\begin{equation}\label{seq_shyr_4}
q(\lambda(O_{K,S}))=\frac{[S_1]}{[Hom_U(\mathcal{Q},\mathbb{G}_m)]}.
\end{equation}
Combining ($\ref{seq_shyr_3b}$) and ($\ref{seq_shyr_4}$) gives us
\begin{equation}
\frac{[Ext^1_U(j_{*}\hat{T},\mathbb{G}_m)]}
{[Ext^1_U(j_{*}\hat{T}',\mathbb{G}_m)]}=
\frac{[Ext^1_U(\mathcal{Q},\mathbb{G}_m)]
[\mathrm{cok}(\hat{\lambda}(H^1_c))]}
{[Ext^2_U(\mathcal{Q},\mathbb{G}_m)][Hom_U(\mathcal{Q},\mathbb{G}_m)] q(\lambda(O_{K,S}))}.
\end{equation}
From $\cite[\mbox{II.2.13}]{Mil06}$ and Artin-Verdier Duality, we deduce 
\[ \frac{[Ext^1_U(\mathcal{Q},\mathbb{G}_m)][H^0_{c}(U,\mathcal{Q})]}
{[Ext^2_U(\mathcal{Q},\mathbb{G}_m)][Hom_U(\mathcal{Q},\mathbb{G}_m)]} 
= \prod_{v\in S_{\infty}}[H^0(K_v,\mathcal{Q}_v)]
=\prod_{v\in S_{\infty}}[H^0(K_v,\hat{T}'')].\]
 Therefore,
\begin{equation}\label{seq_shyr_5}
\frac{[Ext^1_U(j_{*}\hat{T},\mathbb{G}_m)]}
{[Ext^1_U(j_{*}\hat{T}',\mathbb{G}_m)]}=
\frac{\prod_{v\in S_{\infty}}[H^0(K_v,\hat{T}'')]
[\mathrm{cok}(\hat{\lambda}(H^1_c))]}
{ [H^0_{c}(U,\mathcal{Q})]q(\lambda(O_{K,S}))}.
\end{equation}
By diagram chasing, we can see that
\begin{equation*}
\frac{q(\hat{\lambda}(H^1_c))}{q(\hat{\lambda}(H^0_c))}
=\frac{[\mathrm{cok}(\hat{\lambda}(H^1_c))][\ker(\hat{\lambda}(H^0_c))]}
{[H^0_{c}(U,\mathcal{Q})]}.
\end{equation*}
Thus, ($\ref{seq_shyr_5}$) can be rewritten as
\begin{equation}\label{seq_shyr_6}
\frac{[Ext^1_U(j_{*}\hat{T},\mathbb{G}_m)]}
{[Ext^1_U(j_{*}\hat{T}',\mathbb{G}_m)]}=
\frac{\prod_{v\in S_{\infty}}[H^0(K_v,\hat{T}'')]}
{ [\ker(\hat{\lambda}(H^0_c))]q(\lambda(O_{K,S}))}\frac{q(\hat{\lambda}(H^1_c))}{q(\hat{\lambda}(H^0_c))}.
\end{equation}
We have the following commutative diagram
 \begin{equation*}
\xymatrix{
 \underset{v \in S}{\prod}H^{-1}_T(K_v,\hat{T}') \ar[d]^{\hat{\lambda}(H^{-1}_{S})}\ar[r] & H^0_c(U,j_{*}\hat{T}') \ar[r] \ar[d]^{\hat{\lambda}(H^0_{c})} & H^0_{et}(U,j_{*}\hat{T}') \ar[r] \ar[d]^{\hat{\lambda}(H^0_{et})} 
& \underset{v \in S}{\prod}H^0(K_v,\hat{T}')\ar[r] \ar[d]^{\hat{\lambda}(H^{0}_{S})} 
& H^1_c(U,j_{*}\hat{T}') \ar[r] \ar[d]^{\hat{\lambda}(H^1_{c})}
 & 
\alpha({T}') \ar[d]^{\hat{\lambda}(\alpha)}  \\
 \underset{v \in S}{\prod}H^{-1}_T(K_v,\hat{T})  \ar[r] & H^0_c(U,j_{*}\hat{T}) \ar[r]  & H^0_{et}(U,j_{*}\hat{T}) \ar[r]  
& \underset{v \in S}{\prod}H^0(K_v,\hat{T})\ar[r]   & 
H^1_c(U,j_{*}\hat{T}) \ar[r] 
 & \alpha({T})
}
\end{equation*}
As $\hat{\lambda}(H^0_{et})$ is injective, 
$ \ker(\hat{\lambda}(H^0_{c}))\simeq \ker(\hat{\lambda}(H^{-1}_{S}))\simeq
\prod_{v\in S_{\infty}}\ker(\hat{\lambda}(H^{-1}(K_v))
$. Furthermore, 
\begin{eqnarray}\label{seq_shyr_7}
 \frac{q(\hat{\lambda}(H^1_{c}))}{q(\hat{\lambda}(H^0_{c}))}
=\frac{q(\hat{\lambda}(\alpha))q(\hat{\lambda}(H^{0}_{S}))}
{q(\hat{\lambda}(H^0_{et}))q(\hat{\lambda}(H^{-1}_{S}))}
=\frac{q(\hat{\lambda}(\alpha))}
{q(\hat{\lambda}(H^0_{et}))}
\prod_{v\in S_{\infty}} \frac{q(\hat{\lambda}(H^0(K_v))}{q(\hat{\lambda}(H^{-1}(K_v))}
\prod_{v\in S-S_{\infty}} \frac{q(\hat{\lambda}(H^0(K_v))}{q(\hat{\lambda}(H^{-1}(K_v))}.
\end{eqnarray}
Combining ($\ref{seq_shyr_6}$), ($\ref{seq_shyr_7}$) with Lemma $\ref{lemma_qH0/H1_a}$ yields
\begin{equation}\label{seq_shyr_8}
\frac{[Ext^1_U(j_{*}\hat{T},\mathbb{G}_m)]}
{[Ext^1_U(j_{*}\hat{T}',\mathbb{G}_m)]}=
\frac{q(\hat{\lambda}(\alpha))}{q(\lambda(O_{K,S}))q(\hat{\lambda}(H^0_{et}))}
\prod_{v\in S}[\mathrm{cok}(\hat{\lambda}(H^{0}(K_v))].
\end{equation}
Now Theorem $\ref{thm_shyr}$ follows from $(\ref{seq_shyr_0})$, $(\ref{seq_shyr_9})$ and $(\ref{seq_shyr_8})$. 
\end{proof}


\bigskip
\begin{align*}
& \large \mbox{Department of Mathematics, University of Regensburg, 93040 Regensburg, Germany.} \\
& \large \mbox{Email : minh-hoang.tran@mathematik.uni-regensburg.de}
\end{align*}

	\end{document}